\documentclass[12pt]{article}
\usepackage{graphicx}
\usepackage{amsmath,amsthm,amssymb,amsfonts,euscript,enumerate}
\usepackage{amsthm}
\usepackage{color,wrapfig}
\usepackage{pxfonts}
\usepackage{verbatim}
\newtheorem{thm}{Theorem}[section]
\newtheorem{cor}[thm]{Corollary}
\newtheorem{lem}[thm]{Lemma}

\newcommand{\R}{{\mathbb{R}}}
\newcommand{\Z}{{\mathbb{Z}}}

\newcommand{\1}{\partial}
\newcommand{\2}{\overline}
\newcommand{\3}{\varepsilon}
\newcommand{\4}{\widetilde}

\def\ni{\noindent}

\begin{document}
\title{Existence and asymptotic behaviour of solutions of the very fast 
diffusion equation} 
\author{Shu-Yu Hsu\\
Department of Mathematics\\
National Chung Cheng University\\
168 University Road, Min-Hsiung\\
Chia-Yi 621, Taiwan, R.O.C.\\
e-mail: syhsu@math.ccu.edu.tw}
\date{Sept 16, 2011}
\smallbreak \maketitle
\begin{abstract}
Let $n\ge 3$, $0<m\le (n-2)/n$, $p>\max(1,(1-m)n/2)$, and $0\le u_0
\in L_{loc}^p(\R^n)$ satisfy $\liminf_{R\to\infty}R^{-n+\frac{2}{1-m}}
\int_{|x|\le R}u_0\,dx=\infty$. We prove the existence of unique global 
classical solution of $u_t=\frac{n-1}{m}\Delta u^m$, $u>0$, in 
$\R^n\times (0,\infty)$, $u(x,0)=u_0(x)$ in $\R^n$. If in addition 
$0<m<(n-2)/n$ and $u_0(x)\approx A|x|^{-q}$ as $|x|\to\infty$ for some 
constants $A>0$, $q<n/p$, we prove that there exist 
constants $\alpha$, $\beta$, such that the function $v(x,t)
=t^{\alpha}u(t^{\beta}x,t)$ converges uniformly on every compact subset 
of $\R^n$ to the self-similar solution $\psi(x,1)$ of the equation with
$\psi(x,0)=A|x|^{-q}$ as $t\to\infty$. Note that when $m=(n-2)/(n+2)$, 
$n\ge 3$, if $g_{ij}=u^{\frac{4}{n+2}}\delta_{ij}$ is a metric on $\R^n$ 
that evolves by the Yamabe flow $\1 g_{ij}/\1 t=-Rg_{ij}$ with $u(x,0)
=u_0(x)$ in $\R^n$ where $R$ is the scalar curvature, then $u(x,t)$ 
is a global solution of the above fast diffusion equation.
\end{abstract}

\vskip 0.2truein


\vskip 0.2truein
\setcounter{section}{0}

\section{Introduction}
\setcounter{equation}{0}
\setcounter{thm}{0}

Recently there is a lot of study of the equation \cite{A}, \cite{BBDGV}, 
\cite{DK}, \cite{DP}, \cite{DS1}, \cite{DS2}, \cite{Hs2}, \cite{P}, [V1--3],
\begin{equation}\label{fast-diff-eqn}
u_t=\frac{n-1}{m}\Delta u^m.
\end{equation}
The above equation arises in many physical models and in geometry. When 
$m>1$, \eqref{fast-diff-eqn} models the flow of gases through porous media
\cite{A}. When $m=1$, \eqref{fast-diff-eqn} is the heat equation after a 
rescaling. When $m=-1$, the equation \eqref{fast-diff-eqn} arises in
the model of heat conduction in solid hydrogen \cite{R}. As observed by
M.~Del Pino, M.~S\'aez, P.~Daskalopoulos and N.~Sesum \cite{PS}, \cite{DS2}, 
when $m=(n-2)/(n+2)$, $n\ge 3$, if $g_{ij}=u^{\frac{4}{n+2}}\delta_{ij}$ is a 
metric on $\R^n$ that evolves by the Yamabe flow 
\begin{equation}\label{Yamabe-eqn}
\frac{\1 }{\1 t}g_{ij}=-Rg_{ij}
\end{equation}
where $R$ is the scalar curvature, then $u$ satisfies \eqref{fast-diff-eqn}. 
We refer the readers to the book \cite{V3} by J.L.~Vazquez for an introduction
of the equation \eqref{fast-diff-eqn} and the books \cite{DK} by 
P.~Daskalopoulos and C.E.~Kenig, \cite{V2} by J.L.~Vazquez, for the most 
recent results on \eqref{fast-diff-eqn}. 

As observed by L.A.~Peletier in \cite{P} (Theorem 11.3 of \cite{P}) the 
behaviour of the solutions of \eqref{fast-diff-eqn} for $(n-2)/n<m<1$, 
$n\ge 3$, are very different from the behaviour of the solutions of 
\eqref{fast-diff-eqn} for $0<m<(n-2)/n$, $n\ge 3$. For any $0<m<1$, 
$n\ge 3$, let $0\le u_0\in L^1(\R^n)\cap L^p(\R^n)$ for some constant 
$p$ satisfying
\begin{equation}\label{p-cond}
p>\max(1,(1-m)n/2).
\end{equation} 
By Theorem 11.3 of \cite{P} for any $0<m<(n-2)/n$ and $n\ge 3$ there exists 
a constant $T>0$ and a distribution solution $u\in C([0,T];L^1(\R^n))
\cap L^{\infty}([\delta,T]\times \R^n)$ for any $\delta\in (0,T]$ of 
\eqref{fast-diff-eqn} with $u(x,0)=u_0(x)$ on $\R^n$
which vanishes identically at time $T$. For example for any $k>0$ and $T>0$, 
the Barenblatt solution (\cite{DS2}),
\begin{equation}\label{Barenblatt}
B_k(x,t)=\left(\frac{C_{\ast}}{k+(T-t)_+^{\frac{2}{n-2-nm}}|x|^2}
\right)^{\frac{1}{1-m}}(T-t)_+^{\frac{n}{n-2-nm}},
\end{equation}
where $C_{\ast}=2(n-1)(n-2-nm)/(1-m)$ is a classical solution of 
\eqref{fast-diff-eqn} in $\R^n\times (0,T)$ which vanishes identically 
at time $T$. On the other hand an examination of the proof of \cite{HP} 
shows that the existence proof of global solutions of
\begin{equation}\label{fast-diff-IVP}
\left\{\begin{aligned}
u_t=&\frac{n-1}{m}\Delta u^m,u>0,\quad\mbox{ in }\R^n\times (0,T)\\
u(x,0)=&u_0(x)\qquad\qquad\qquad\mbox{ in }\R^n
\end{aligned}\right.
\end{equation} 
is valid only for $(n-2)/n<m<1$ and $n\ge 3$. Then by the result of 
\cite{HP} for $n\ge 3$, $(n-2)/n<m<1$, and any $0\le u_0\in L_{loc}^1
(\R^n)$, $u_0\not\equiv 0$,  there exists a unique global 
solution of \eqref{fast-diff-IVP} in $\R^n\times (0,\infty)$. 

In this paper we will prove that if $n\ge 3$, $0<m\le (n-2)/n$, and
$0\le u_0\in L_{loc}^p(\R^n)$ for some constant $p$ satisfying 
\eqref{p-cond} and 
\begin{equation}\label{initial-value-average-infty}
\liminf_{R\to\infty}\frac{1}{R^{n-\frac{2}{1-m}}}\int_{|x|\le R}u_0\,dx
=\infty,
\end{equation}
then \eqref{fast-diff-IVP} has a global solution $u$ in 
$\R^n\times (0,\infty)$. When $u_0$ is radially symmetric, the 
condition \eqref{initial-value-average-infty} is also shown in
\cite{DP} to imply global existence of solution of 
\eqref{fast-diff-IVP} for $m<0$. On the other hand when $m>1$,
by the result of Aronson and Caffarelli \cite{AC}, if \eqref{fast-diff-IVP}  
has a global solution in $\R^n\times (0,\infty)$, then
\begin{equation*}
\limsup_{R\to\infty}\frac{1}{R^{n-\frac{2}{1-m}}}\int_{|x|\le R}u_0\,dx
=0.
\end{equation*}
Thus our result extends the result of \cite{AC} and \cite{DP}. Let
\begin{equation}\label{alpha-beta-defn}
\alpha=\frac{q}{2-q(1-m)}\quad\mbox{ and }\quad\beta=\frac{1}{2-q(1-m)}. 
\end{equation}
We will also prove that when $u_0(x)\approx A|x|^{-q}$ for some constants 
$A>0$, $q<n/p$, as $|x|\to\infty$ where $p$ satisfies \eqref{p-cond}, 
then under some mild condition on $u_0$ the rescaled function 
\begin{equation}\label{rescaled-soln}
v(x,t)=t^{\alpha}u(t^{\beta}x,t)
\end{equation}
converges uniformly on every compact subset of $\R^n$ to the 
self-similar solution $\psi(x,1)$ of \eqref{fast-diff-eqn} with
$\psi (x,0)=A|x|^{-q}$ as $t\to\infty$. The function $\4{v}(x)=\psi (x,1)$ 
is radially symmetric and satisfies the elliptic equation
\begin{equation}\label{elliptic}
\frac{n-1}{m}\Delta\4{v}^m+\alpha\4{v}+\beta x\cdot\nabla\4{v}=0, \4{v}>0,
\quad\mbox{ in }\R^n.
\end{equation}
The main results we obtain in this paper are the following.

\begin{thm}\label{existence-finite-time}
Let $n\ge 3$, $0<m\le (n-2)/n$, and let $p$ satisfy \eqref{p-cond}. 
Then there exists a constant $C_1>0$ such that if $0\le u_0\in 
L_{loc}^p(\R^n)$ satisfies 
\begin{equation}\label{initial-value-average-lower-bd-0}
\liminf_{R\to\infty}\frac{1}{R^{n-\frac{2}{1-m}}}\int_{|x|\le R}u_0\,dx
\ge C_1T^{\frac{1}{1-m}}
\end{equation} 
for some constant $T>0$, then there exists a unique positive solution 
$u\in C^{\infty}(\R^n\times (0,\infty))$ of 
\eqref{fast-diff-IVP} in $\R^n\times (0,T)$ which satisfy 
\begin{equation}\label{aronson-benilan}
u_t\le\frac{u}{(1-m)t}
\end{equation}
in $\R^n\times (0,T)$.
\end{thm}

\begin{thm}\label{existence-thm}
Let $n\ge 3$, $0<m\le (n-2)/n$, and let $p$ satisfy \eqref{p-cond}. Suppose 
$0\le u_0\in L_{loc}^p(\R^n)$ satisfies \eqref{initial-value-average-infty}. 
Then \eqref{fast-diff-IVP} has a unique global positive solution $u\in C^{\infty}
(\R^n\times (0,\infty))$ which satisfies \eqref{aronson-benilan} in 
$\R^n\times (0,\infty)$.
\end{thm}

\begin{thm}\label{asymptotic-thm}
Let $n\ge 3$, $0<m<(n-2)/n$, and $q<n/p$ for some constant $p$ satisfying 
\eqref{p-cond}. Suppose $0\le u_0\in L_{loc}^p(\R^n)$ satisfies 
$u_0=\2{u}_0+\phi$ where $0\le\2{u}_0\in L_{loc}^p(\R^n)$ and $\phi\in 
L^1(\R^n)\cap L^p(\R^n)$ such that 
\begin{equation}\label{u0-bar-infty}
\lim_{|x|\to\infty}|x|^q\2{u}_0(x)=A
\end{equation}
for some constant $A>0$. Let $\alpha$ and $\beta$ be given by 
\eqref{alpha-beta-defn}. Let $u$ be the unique global solution of
\eqref{fast-diff-IVP} in $\R^n\times (0,\infty)$ given by 
Corollary~\ref{global-existence-cor} which satisfies 
\eqref{aronson-benilan} in $\R^n\times (0,\infty)$. Let $v$ be given 
by \eqref{rescaled-soln}.
Then the rescaled function $v$ converges uniformly on every 
compact subset of $\R^n$ to the unique radially symmetric self-similar 
solution $\psi(x,1)$ of \eqref{fast-diff-eqn} with $\psi(x,0)=A|x|^{-q}$
as $t\to\infty$.
\end{thm}

\begin{cor}\label{Yamabe-cor}
Let $n\ge 3$ and $m=(n-2)/(n+2)$. Let $p$, $q$, $\alpha$, $\beta$, $u_0$, 
$v$, be as in Theorem~\ref{asymptotic-thm}. Suppose the metric $g_{ij}
=u^{4/(n+2)}\delta_{ij}$ evolves by the Yamabe flow \eqref{Yamabe-eqn}
with $u(x,0)=u_0(x)$ on $\R^n$. Then $v(x,t)$ converges uniformly on 
every compact subset of $\R^n$ to the unique radially symmetric 
self-similar solution $\psi(x,1)$ of \eqref{fast-diff-eqn} with 
$\psi(x,0)=A|x|^{-q}$ as $t\to\infty$.
\end{cor}

Note that the Barenblatt solution $B_k(x,t)$ given by \eqref{Barenblatt}
satisfies 
$$
\lim_{R\to\infty}\frac{1}{R^{n-\frac{2}{1-m}}}\int_{|x|\le R}B_k(x,0)\,dx
=\frac{\omega_nC_{\ast}^{\frac{1}{1-m}}T^{\frac{1}{1-m}}}{n-2(1-m)^{-1}}
$$
where $\omega_n$ is the surface area of the unit sphere $S^{n-1}$ in 
$\R^n$ and it vanishes in a finite time $T$. Hence 
Theorem~\ref{existence-finite-time} is sharp.

The plan of the paper is as follows. In section two we will prove the 
existence of global solutions of \eqref{fast-diff-IVP}. In section 
three we will prove the asymptotic large time behaviour of the global
solution of \eqref{fast-diff-IVP}.

We first start will some definitions (cf. \cite{Hu2}). For any $x_0\in\R^n$, 
$R>0$, $T>0$, let $B_R(x_0)=\{x\in\R^n:|x-x_0|<R\}$, $B_R=B_R(0)$, $Q_R(x_0)
=B_R(x_0)\times (0,\infty)$, $Q_R=Q_R(0)$, $Q_R^T(x_0)=B_R(x_0)\times 
(0,T)$ and $Q_R^T=Q_R^T(0)$. For any set $A\subset\R^n$, we let 
$\chi_A$ be the characteristic function of the set $A$. 
For any domain $\Omega\subset\R^n$, $T>0$, $0<m<1$, we say that $u$ is 
a solution (subsolution, supersolution) of \eqref{fast-diff-eqn} in 
$\Omega\times (0,T)$ if $u>0$ in $\Omega\times (0,T)$ and is a classical 
solution (subsolution, supersolution) of \eqref{fast-diff-eqn}
in $\Omega\times (0,T)$. For any $0\le u_0\in L_{loc}^1(\Omega)$
we say that a solution $u$ of \eqref{fast-diff-eqn} in $\Omega\times 
(0,T)$ has initial value $u_0$ if $\|u(\cdot,t)-u_0\|_{L^1(K)}\to 0$
as $t\to 0$ for any compact subset $K\subset\Omega$.  

For any bounded smooth domain $\Omega\subset\R^n$, 
$0\le u_0\in L_{loc}^1(\Omega)$, and $0\le g\in 
L^1(\1\Omega\times (0,T))$, we say that $u$ is a 
solution of the Dirichlet problem
\begin{equation}\label{Dirichlet-problem-g}
\left\{\begin{aligned}
u_t=&\frac{n-1}{m}\Delta u^m\quad\mbox{ in }\Omega\times (0,T)\\
u=&g\quad\qquad\quad\,\,\,\mbox{ on }\1\Omega\times (0,T)\\
u(x,0)=&u_0(x)\qquad\quad\mbox{ in }\Omega
\end{aligned}\right.
\end{equation}
if $u$ is a positive classical solution of \eqref{fast-diff-eqn} in 
$\Omega\times (0,T)$ with initial value $u_0$ and satisfies
\begin{equation}\label{integral-eqn}
\int_{t_1}^{t_2}\int_{\Omega}\left(\frac{n-1}{m}u^m\Delta\eta+u\eta_t\right)
\,dx\,ds=\frac{n-1}{m}\int_{t_1}^{t_2}\int_{\1\Omega}g^m\frac{\1\eta}{\1 n}
\,d\sigma dt+\int_{\Omega}u(x,t_2)\,dx-\int_{\Omega}u(x,t_1)\,dx
\end{equation}
for any $0<t_1<t_2<T$ and $\eta\in C^2(\2{\Omega}\times (0,T))$
satisfying $\eta=0$ on $\1\Omega\times (0,T)$ where $\1/\1 n$ is the 
exterior normal derivative on $\1\Omega$. 
We say that $u$ is a weak solution of the Dirichlet problem 
\eqref{Dirichlet-problem-g} if $0\le u\in C([0,T_R);L^1(\Omega))$ 
satisfies \eqref{integral-eqn} for any $0<t_1<t_2<T$ and 
$\eta\in C^2(\2{\Omega}\times (0,T))$ satisfying $\eta=0$ on 
$\1\Omega\times (0,T)$ and $u$ has initial value $u_0$. 

We say that $u$ is a solution of the Dirichlet problem
\begin{equation}\label{infinity-bd-problem}
\left\{\begin{aligned}
u_t=&\frac{n-1}{m}\Delta u^m, u>0,\quad\mbox{ in }
\Omega\times (0,\infty)\\
u(x,t)=&\infty\qquad\qquad\qquad\quad\,\,\mbox{ on }
\1\Omega\times (0,\infty)\\
u(x,0)=&u_0(x)\qquad\qquad\qquad\mbox{ in }\Omega
\end{aligned}\right.
\end{equation}
if $u$ is a positive classical solution of \eqref{fast-diff-eqn} in 
$\Omega\times (0,\infty)$ with initial value $u_0$ and 
$$
\lim_{\substack{(y,s)\to (x,t)\\(y,s)\in\Omega\times (0,\infty)}}u(y,s)
=\infty\quad\forall (x,t)\in\1\Omega\times (0,\infty)
$$
We say that $u$ is a solution (subsolution, supersolution) of 
\eqref{fast-diff-eqn} in $\R^n\times (0,T)$ if $u>0$ in $\R^n\times (0,T)$ 
and is a classical solution (subsolution, supersolution) of 
\eqref{fast-diff-eqn} in $\R^n\times (0,T)$. For any $0\le u_0\in 
L_{loc}^1(\R^n)$ we say that a solution $u$ of \eqref{fast-diff-eqn} in 
$\R^n\times (0,T)$ has initial value $u_0$ if $\|u(\cdot,t)-u_0\|_{L^1(K)}
\to 0$ as $t\to 0$ for any compact subset $K\subset\R^n$. 

We will assume that $n\ge 3$, $0<m\le (n-2)/n$, and \eqref{p-cond} hold 
for the rest of the paper.

\section{Existence of global solutions}
\setcounter{equation}{0}
\setcounter{thm}{0}

In this section we will prove the existence of global
solutions of \eqref{fast-diff-IVP}. We first extend some results of \cite{Hu2}.
We first observe that by \eqref{p-cond},
$$
\frac{2}{n}+\frac{\alpha+m}{\alpha+1}
\ge\frac{2}{n}+\frac{p-1+m}{p}>1\quad\forall\alpha\ge p-1.
$$ 
Hence there exists a constant 
\begin{equation}\label{k'-defn}
1<k'<\frac{1}{q}+\frac{\alpha+m}{\alpha+1}\quad\forall\alpha\ge p-1
\end{equation}
where $q=n/2$. Let 
\begin{equation}\label{k-defn}
k=\left(\frac{\alpha+m}{\alpha+1}\right)k'.
\end{equation}
Then by the same argument as the proof of Theorem 1.6 of \cite{Hu2} but with 
the $\alpha_0$, $k$, $k'$, there being replaced by $\alpha_0=p-1$ and $k$, 
$k'$, given by \eqref{k'-defn}, \eqref{k-defn}, the proof of Theorem 1.6 of 
\cite{Hu2} remains valid for $n\ge 3$, $0<m\le (n-2)/n$, and $p$ satisfying 
\eqref{p-cond}. Hence we have the following theorem.
 
\begin{thm}\label{L-infinity-Lp-bd1}(cf. Theorem 1.6 of \cite{Hu2})
Let $n\ge 3$, $0<m\le (n-2)/n$, and let $p$ satisfy \eqref{p-cond}. 
Suppose $u$ is a solution of \eqref{fast-diff-eqn} in $Q^{\ast}
=\{(x,t)\in\R^{n+1}:|x|<2,-4<t<0\}$. Let $Q=\{(x,t)\in\R^{n+1}:|x|<1,
-1<t<0\}$. Then there exist constants $C>0$ and $\theta>0$ such that
$$
\|u\|_{L^{\infty}(Q)}\le C\left(1+\iint_{Q^{\ast}}u^p\,dx\,dt\right)^{\frac{\theta}{p}}.
$$
\end{thm}
We next observe that the result of Lemma 1.7 and Lemma 1.9 of \cite{Hu2} 
remains valid for any $n\ge 3$, $0<m<1$, $p$ satisfying \eqref{p-cond},
and $0\le g\in L^{\infty}(\1\Omega\times (0,\infty))$ where $\Omega$ is a 
bounded smooth domain. By Lemma 1.7 and Lemma 1.9 of \cite{Hu2}, 
Theorem~\ref{L-infinity-Lp-bd1}, and an argument similar to that of 
\cite{Hu2} we have the following three results.

\begin{cor}\label{L-infinity-initial-Lp-bd1}(cf. Corollary 1.8 of \cite{Hu2})
Let $n\ge 3$, $0<m\le (n-2)/n$, and let $p$ satisfy \eqref{p-cond}.
Suppose $u$ is a solution of \eqref{fast-diff-eqn} in $\Omega\times (0,T)$
with initial value $0\le u_0\in L_{loc}^p(\Omega)$. Then for any $B_{R_1}(x_0)
\subset\2{B_{R_2}(x_0)}\subset\Omega$ and $0<t_1<T$ there exist constants 
$C>0$ and $\theta>0$ such that
\begin{equation*}
\|u\|_{L^{\infty}(B_{R_1}(x_0)\times [t_1,T))}
\le C\left(1+\int_{B_{R_2}(x_0)}u_0^p\,dx\right)^{\frac{\theta}{p}}.
\end{equation*}
\end{cor}

\begin{thm}\label{L-infinity-initial-Lp-bd2}(cf. Corollary 1.11 of \cite{Hu2})
Let $n\ge 3$, $0<m\le (n-2)/n$, and let $p$ satisfy \eqref{p-cond}. 
Let $\Omega\subset\R^n$ be a bounded smooth domain and let $0\le g
\in L^{\infty}(\1\Omega\times (0,T))$. Suppose $u$ is a solution
of \eqref{Dirichlet-problem-g} in $\Omega\times (0,T)$. Then for any 
$0<t_1<T$ there exist constants $C>0$ and $\theta>0$ such that
\begin{equation*}
\|u\|_{L^{\infty}(\2{\Omega}\times [t_1,T))}
\le C\left(k_g^p|\Omega|+\int_{\Omega}u_0^p\,dx\right)^{\frac{\theta}{p}}
+k_g
\end{equation*}
where $k_g=\max(1,\|g\|_{L^{\infty}(\1\Omega\times (0,T))})$.
\end{thm}

\begin{thm}\label{existence-infinity-bd-problem}(cf. Theorem 2.5 and 
Theorem 2.11 of \cite{Hu2}) 
Let $n\ge 3$, $0<m\le (n-2)/n$, and let $p$ satisfy \eqref{p-cond}. 
Let $\Omega\subset\R^n$ be a bounded smooth domain. Then there exists a 
minimal solution $u$ for \eqref{infinity-bd-problem} which satisfies
\eqref{aronson-benilan} in $\Omega\subset\R^n$. When $\Omega$ is a 
smooth bounded star shape domain, the solution $u$ for 
\eqref{infinity-bd-problem} is unique.  
\end{thm}

Note that a result similar to Corollary~\ref{L-infinity-initial-Lp-bd1}
is also obtained recently by M.~Bonforte and J.L.~Vazquez (Theorem 2.1
of \cite{BV}). 

\begin{lem}\label{positive-lem}
Let $n\ge 3$, $0<m\le (n-2)/n$, and let $p$ satisfy \eqref{p-cond}. 
Suppose $0\le u_0\in L_{loc}^p(\R^n)$. For any $R>0$ let 
$0\le w_R\in C([0,T_R);L^1(B_{5R}))$ be the unique weak solution of 
\begin{equation}\label{Dirichlet-problem-1}
\left\{\begin{aligned}
w_t=&\frac{n-1}{m}\Delta w^m\quad\,\,\,\mbox{ in }Q_{5R}^{T_R}\\
w=&0\quad\qquad\qquad\,\mbox{ on }\1 B_{5R}\times (0,T_R)\\
w(x,0)=&u_0(x)\chi_{B_{2R}}(x)\quad\mbox{ in }B_{5R}
\end{aligned}\right.
\end{equation}
given by \cite{BC} (cf. P.537 of \cite{BV}) which extincts in a 
finite time $T_R>0$. Then there exists a constant $C_1>0$ such that if 
\eqref{initial-value-average-lower-bd-0} holds for some constant $T>0$,
then there exist constants $0<\delta<1$, $R_0>0$, such that for any 
$R\ge R_0$, $T_R>T$,  $w_R$ is continuous on $\2{B_{\delta R}}\times (0,T]$
and
\begin{equation}\label{weak-soln-lower-bd}
\inf_{\2{B_{\delta R}}\times [t_1,T]}w_R>0\quad\forall 0<t_1<T, R\ge R_0.
\end{equation} 
\end{lem}
\begin{proof}
Suppose there exists $T>0$ such that $u_0$ satisfies 
\eqref{initial-value-average-lower-bd-0} for some constant $C_1>0$
to be determined later.  
For any $\3>0$ let $w_{R,\3}$ be the solution of 
\begin{equation*}\label{Dirichlet-problem-2}
\left\{\begin{aligned}
w_t=&\frac{n-1}{m}\Delta w^m\qquad\quad\,\,\mbox{ in }Q_{5R}\\
w=&\3\quad\qquad\qquad\qquad\,\mbox{ on }\1 B_{5R}\times (0,\infty)\\
w(x,0)=&u_0(x)\chi_{B_{2R}}(x)+\3\quad\mbox{ in }B_{5R}
\end{aligned}\right.
\end{equation*}
Then $w_{R,\3}\ge w_{R,\3'}\ge\3'$ in $Q_{5R}$ for any $\3>\3'>0$ and $w_{R,\3}$ 
decreases to the weak solution $w_R$ of \eqref{Dirichlet-problem-1} as 
$\3\to 0$. By \eqref{initial-value-average-lower-bd-0} there exists a 
constant $0<\delta<1$ such that
\begin{equation*}
(1-\delta)^{n-\frac{2}{1-m}}\liminf_{R\to\infty}\frac{1}{R^{n-\frac{2}{1-m}}}
\int_{|x|\le R}u_0\,dx>(3C_1/4)T^{\frac{1}{1-m}}
\end{equation*} 
Let $|x_0|\le\delta R$. Then
\begin{align}\label{initial-value-average-lower-bd-3}
\liminf_{R\to\infty}\frac{1}{R^{n-\frac{2}{1-m}}}
\int_{|x-x_0|\le R}u_0\,dx
\ge&\liminf_{R\to\infty}\frac{1}{R^{n-\frac{2}{1-m}}}
\int_{|x|\le(1-\delta )R}u_0\,dx\notag\\
\ge&(1-\delta)^{n-\frac{2}{1-m}}\liminf_{R\to\infty}\frac{1}{R^{n-\frac{2}{1-m}}}
\int_{|x|\le R}u_0\,dx>(3C_1/4)T^{\frac{1}{1-m}}.
\end{align} 
Let $0\le v_R\in C([0,\4{T}_R);L^1(B_{3R}))$ be the unique weak solution of 
\begin{equation}\label{Dirichlet-problem-3}
\left\{\begin{aligned}
v_t=&\frac{n-1}{m}\Delta v^m\qquad\,\,\,\mbox{ in }Q_{3R}^{\4{T}_R}(x_0)\\
v=&0\quad\qquad\qquad\quad\,\mbox{ on }\1 B_{3R}(x_0)\times (0,\4{T}_R)\\
v(x,0)=&u_0(x)\chi_{B_R(x_0)}(x)\quad\mbox{ in }B_{3R}(x_0)
\end{aligned}\right.
\end{equation}
given by \cite{BC} (cf. P.537 of \cite{BV}) which extincts in a finite 
time $\4{T}_R>0$ and let $v_{R,\3}$ be the solution of 
\begin{equation*}\label{Dirichlet-problem-0}
\left\{\begin{aligned}
v_t=&\frac{n-1}{m}\Delta v^m\qquad\qquad\,\,\mbox{ in }Q_{3R}(x_0)\\
v=&\3\quad\qquad\qquad\qquad\quad\mbox{ on }\1 B_{3R}(x_0)\times (0,\infty)\\
v(x,0)=&u_0(x)\chi_{B_R(x_0)}(x)+\3\quad\mbox{ in }B_{3R}(x_0)
\end{aligned}\right.
\end{equation*}
for any $\3>0$. Since $v_{R',\3}\ge\3$ in $Q_{3R'}(x_0)$ for any $R'>0$, 
by the maximum principle (cf. Lemma 2.3 of \cite{DaK}),
\begin{equation}\label{vr-compare}
v_{R',\3}\ge v_{R,\3}\quad\mbox{ in }Q_{3R}(x_0)\quad\mbox{ for any }R'>R>0
\end{equation}
and 
\begin{equation}\label{wr-vr-epsilon-compare}
w_{R,\3}\ge v_{R,\3}\quad\mbox{ in }Q_{3R}(x_0)\quad\forall R>0,\3>0.
\end{equation}
Since $v_{R,\3}$ decreases to $v_R$ as $\3\to 0$, 
by \eqref{wr-vr-epsilon-compare},
\begin{equation}\label{wr-vr-compare}
w_{R,\3}\ge v_R\quad\mbox{ in }Q_{3R}^{\4{T}_R}(x_0).
\end{equation}
Letting $\3\to 0$ in \eqref{vr-compare} we get $v_{R'}\ge v_R$ in 
$Q_{3R}^{\4{T}_R}(x_0)$ for any $R'>R>0$. Hence $\4{T}_{R'}\ge\4{T}_R$ for 
any $R'>R>0$. 
By (1.18) of \cite{BV} there exists a constant $C_2>0$ such that 
\begin{equation}\label{extinction-time-l-bd-a}
\4{T}_R\ge C_2R^2
\left[\frac{\int_{B_R(x_0)}u_0\,dx}{Vol(B_{3R}\setminus B_{2R})}
\right]^{1-m}\quad\forall R>0.
\end{equation}
Letting $R\to\infty$ in \eqref{extinction-time-l-bd-a}, by 
\eqref{initial-value-average-lower-bd-3} we get
$$
\lim_{R\to\infty}\4{T}_R^{\frac{1}{1-m}}\ge\frac{3nC_1C_2^{\frac{1}{1-m}}}
{4(3^n-2^n)\omega_n}T^{\frac{1}{1-m}}
$$
where $\omega_n$ is the surface area of the unit sphere $S^{n-1}$ in $\R^n$. 
We now choose $C_1\ge 4(3^n-2^n)\omega_n/(nC_2^{1/(1-m)})$. Then 
\begin{equation}\label{Tr-limit-0}
\lim_{R\to\infty}\4{T}_R>T. 
\end{equation}
By (1.28) of \cite{BV} there exist constants $C_3>0$, $C_4>0$, such that
\begin{equation}\label{Harnack}
\frac{1}{R^{n-\frac{2}{1-m}}}\int_{B_{R}(x_0)}u_0\,dx\le C_3t^{\frac{1}{1-m}}
+C_4\4{T}_R^{\frac{1}{1-m}}R^{\frac{2m}{1-m}}t^{-\frac{m}{1-m}}v_R^m(t,x_0)
\quad\forall 0<t<\4{T}_R, R>0.
\end{equation}
Let $C_1=\max (3C_3,4(3^n-2^n)\omega_n/(nC_2^{1/(1-m)}))$. By
\eqref{initial-value-average-lower-bd-3}, 
\begin{equation}\label{initial-value-average-lower-bd-2}
\liminf_{R\to\infty}\frac{1}{R^{n-\frac{2}{1-m}}}\int_{B_R(x_0)}u_0\,dx
\ge (9C_3/4)T^{\frac{1}{1-m}}.
\end{equation} 
Then by \eqref{Tr-limit-0} and \eqref{initial-value-average-lower-bd-2} 
there exists a constant $R_0>0$ such that $\4{T}_R>T$ for any $R\ge R_0$
and 
\begin{equation}\label{average-lower-bd}
\frac{1}{R^{n-\frac{2}{1-m}}}\int_{B_R(x_0)}u_0\,dx>2C_3T^{\frac{1}{1-m}}
\quad\forall R\ge R_0.
\end{equation}
By \eqref{wr-vr-compare}, \eqref{Harnack}, and \eqref{average-lower-bd},
$$
C_4\4{T}_R^{\frac{1}{1-m}}R^{\frac{2m}{1-m}}t^{-\frac{m}{1-m}}w_{R,\3}^m(x_0,t)
\ge C_4\4{T}_R^{\frac{1}{1-m}}R^{\frac{2m}{1-m}}t_0^{-\frac{m}{1-m}}v_R^m(x_0,t)
>C_3(2T^{\frac{1}{1-m}}-t^{\frac{1}{1-m}})>0
$$
for any $R\ge R_0$, $\3>0$, $|x_0|\le\delta R$ and $0<t\le T$. Hence
\begin{equation}\label{wr-epsilon-lower-bd}
\inf_{\2{B_{\delta R}}\times [t_1,T]}w_{R,\3}^m(x,t)
\ge (C_3/C_4)\4{T}_R^{-\frac{1}{1-m}}R^{-\frac{2m}{1-m}}t_1^{\frac{m}{1-m}}
(2T^{\frac{1}{1-m}}-t_1^{\frac{1}{1-m}})>0\quad\forall 0<t_1<T,R\ge R_0,
\3>0.
\end{equation}
By Corollary~\ref{L-infinity-initial-Lp-bd1} for any 
$R\ge R_0$, $T\ge t_1>0$, there exist constants $C_R>0$ and $\theta>0$ such that
\begin{equation}\label{wr-epsilon-upper-bd}
\|w_{R,\3}\|_{L^{\infty}(\2{B_{\delta R}}\times [t_1,T])}
\le C_R\left(1+\int_{B_{2R}}u_0^p\,dx\right)^{\frac{\theta}{p}}\quad\forall 0<\3<1.
\end{equation}
Let $R\ge R_0$. By \eqref{wr-epsilon-lower-bd} and 
\eqref{wr-epsilon-upper-bd} the equation \eqref{fast-diff-eqn} for 
the sequence $\{w_{R,\3}\}_{0<\3<1}$ is uniformly parabolic on 
$\2{B_{\delta R}}\times [t_1,T]$. By reducing $\delta$ if necessary by the
parabolic Schauder estimates \cite{LSU} the sequence $\{w_{R,\3}\}_{0<\3<1}$
is equi-Holder continuous in $C^2(K)$ on every compact subset 
$K\subset\2{B_{\delta R}}\times [t_1,T]$. 
Hence by the Ascoli Theorem $w_{R,\3}$ converges in $C^2$ on every compact 
subset of $\2{B_{\delta R}}\times (0,T]$ to $w_R$ as $\3\to 0$. Hence 
$w_R\in C^2(\2{B_{\delta R}}\times (0,T])$. Letting $\3\to 0$ in 
\eqref{wr-epsilon-lower-bd}, we get \eqref{weak-soln-lower-bd}
and the lemma follows.
\end{proof}

{\ni{\it Proof of Theorem~\ref{existence-finite-time}}:}
Uniqueness of solution of \eqref{fast-diff-IVP} is given by Theorem 2.3 
of \cite{HP}. Hence we will only need to prove existence of solution of 
\eqref{fast-diff-IVP}. We will give two different methods of construction 
of solution of \eqref{fast-diff-IVP}. Let $C_1>0$ be as in 
Lemma~\ref{positive-lem}.

\noindent $\underline{\text{\bf First method}}$:
For any $R>0$ by Theorem~\ref{existence-infinity-bd-problem} 
\eqref{infinity-bd-problem} has a unique solution $u_R$ in $Q_R$ with 
$\Omega=B_R$ which satisfies \eqref{aronson-benilan} in $Q_R$. 
By the maximal principle (Lemma 2.9 of \cite{Hu2}),
\begin{equation}\label{ur-monotone}
u_R\ge u_{R'}>0\quad\mbox{ in }Q_R\quad\forall R'>R>0. 
\end{equation}
Hence
$$
0\le u(x,t)=\lim_{R\to 0}u_R(x,t)\quad\forall (x,t)\in\R^n\times (0,\infty)
$$
exists. By Corollary~\ref{L-infinity-initial-Lp-bd1} for any 
$R_1>0$, $t_2>t_1>0$, there exist constants $C>0$ and $\theta>0$ such that
\begin{align}
&\|u_R\|_{L^{\infty}(\2{B_{R_1}(x_0)}\times [t_1,t_2])}
\le C_{R_1,t_1,t_2}
:=C\left(1+\int_{B_{2R_1}(x_0)}u_0^p\,dx\right)^{\frac{\theta}{p}}\quad\forall
R>3R_1+1\label{ur-upper-bd}\\
\Rightarrow\quad&\|u\|_{L^{\infty}(\2{B_{R_1}(x_0)}\times  [t_1,t_2])}
\le C_{R_1,t_1,t_2}\qquad\qquad\qquad\qquad\qquad\qquad\mbox{ as }R\to\infty.
\label{u-upper-bd}
\end{align}
We claim that
\begin{equation}\label{lower-bd}
\inf_{\substack{\2{B_R}\times [t_1,T]}}u(x,t)>0\quad\forall R>0,
T>t_1>0.  
\end{equation}
Suppose the claim is false. Let $R_0$, $w_R$, $T_R$, be as in 
Lemma~\ref{positive-lem}. For any $\3>0$ let $w_{R,\3}$ be as in the proof
of Lemma~\ref{positive-lem}. By the maximum principle (Lemma 2.3 of 
\cite{DaK} and Lemma 2.9 of \cite{Hu2}),
\begin{align}
&w_{R_1',\3}\ge w_{R_1,\3}\quad\mbox{ in }Q_{5R_1}^{T_{R_1}}\quad\forall
R_1'>R_1>0,\3>0\notag\\
\Rightarrow\quad&w_{R_1'}\ge w_{R_1}\quad\mbox{ in }Q_{5R_1}^{T_{R_1}}
\quad\forall R_1'>R_1>0\quad\mbox{ as }\3\to 0\label{wr-increase}
\end{align} 
and
\begin{equation}\label{ur-wr-compare}
u_{R}\ge w_{R_1}\quad\mbox{ in }Q_{5R_1}^{T_{R_1}}\quad\forall
R>5R_1>0.
\end{equation} 
Letting $R\to\infty$ in \eqref{ur-wr-compare},
\begin{equation}\label{u-wr-compare}
u\ge w_{R_1}\quad\mbox{ in }Q_{5R_1}^{T_{R_1}}\quad\forall R_1>0.
\end{equation} 
By \eqref{weak-soln-lower-bd} and \eqref{u-wr-compare} the claim 
\eqref{lower-bd} holds. By \eqref{ur-upper-bd}, \eqref{lower-bd}, 
the equation \eqref{fast-diff-eqn} for the sequence $\{u_R\}_{R>R_0}$
are uniformly parabolic on every compact subset of $\R^n\times (0,T]$.
By the Schauder estimates for parabolic equations \cite{LSU}
the sequence $\{u_R\}_{R>R_0}$ are equi-Holder continuous on 
every compact subset of $\R^n\times (0,T]$. Hence by the Ascoli
Theorem and \eqref{ur-monotone} the sequence $u_R$ decreases and 
converges uniformly on every compact subset of $\R^n\times (0,T]$ 
to a solution $u$ of \eqref{fast-diff-eqn} in $\R^n\times (0,T)$
as $R\to\infty$. Then $u$ also satisfies \eqref{aronson-benilan} in 
$\R^n\times (0,T)$.

We will now prove that $u$ has initial value $u_0$. By the Kato 
inequality \cite{K} (cf. \cite{DS1}) and an argument similar to the 
proof of Lemma 3.1 of \cite{HP} and Lemma 2.3 of \cite{Hu2},
\begin{align}\label{l1-compare}
\left(\int_{B_{R_1}}|u_R-u_{R'}|(x,t)\,dx\right)^{1-m}
\le&\left(\int_{B_{2R_1}}|u_R-u_{R'}|(x,0)\,dx\right)^{1-m}
+CR_1^{n(1-m)-2}t\notag\\
=&CR_1^{n(1-m)-2}t\quad\forall R>2R_1,R'>2R_1>0,t>0\notag\\
\Rightarrow\qquad\qquad\int_{B_{R_1}}|u_R-u|(x,t)\,dx
\le&C'R_1^{n-\frac{2}{1-m}}t\quad\forall R>2R_1>0,t>0\quad\mbox{ as }
R'\to\infty.
\end{align}
Hence
\begin{align}\label{u-u0-L1-bd}
\int_{B_{R_1}}|u(x,t)-u_0(x)|\,dx
\le&\int_{B_{R_1}}|u_R(x,t)-u_0(x)|\,dx+\int_{B_{R_1}}|u-u_R|(x,t)\,dx
\notag\\
\le&\int_{B_{R_1}}|u_R(x,t)-u_0(x)|\,dx+C'R_1^{n-\frac{2}{1-m}}t
\quad\forall R>2R_1>0,t>0.
\end{align}
Letting $t\to 0$ in \eqref{u-u0-L1-bd},
\begin{align*}
\limsup_{t\to 0}\int_{B_{R_1}}|u(x,t)-u_0(x)|\,dx\le&0\quad\forall 
R_1>0\\
\Rightarrow\qquad\quad\lim_{t\to 0}\int_{B_{R_1}}|u(x,t)-u_0(x)|\,dx
=&0\quad\forall R_1>0.
\end{align*}
Hence $u$ is the unique solution of \eqref{fast-diff-IVP} in 
$\R^n\times (0,T)$.

\noindent $\underline{\text{\bf Second method}}$: By 
\eqref{weak-soln-lower-bd}, \eqref{u-upper-bd}, \eqref{wr-increase},
and \eqref{u-wr-compare}, the sequence
$\{w_R:R>R_0\}$ are uniformly bounded below and above
on every compact subset of $\R^n\times (0,T)$. Hence the
equation \eqref{fast-diff-eqn} for $\{w_R:R>R_0\}$ are
uniformly parabolic on every compact subset of $\R^n\times (0,T)$.
By the parabolic Schauder estimates \cite{LSU} the sequence
$\{w_R:R>R_0\}$ are equi-Holder continuous in $C^2$ on 
on every compact subset of $\R^n\times (0,T)$. Hence
$w_R$ increases and converges uniformly on every compact subset of 
$\R^n\times (0,T)$ to a solution $w$ of \eqref{fast-diff-eqn}
as $R\to\infty$.
By an argument similar to the first method of proof $w$ has
initial value $u_0$. Hence by the uniqueness of solution
$w=u$ is the unique solution of \eqref{fast-diff-IVP} in 
$\R^n\times (0,T)$.
\hfill$\square$\vspace{6pt}

\begin{thm}\label{existence-finite-time-2}
Let $n\ge 3$, $0<m\le (n-2)/n$, and let $p$ satisfy \eqref{p-cond}. 
Let $C_1>0$ be as in Lemma~\ref{positive-lem}. Suppose $0\le u_0\in 
L_{loc}^p(\R^n)$ satisfies \eqref{initial-value-average-lower-bd-0}
for some constant $T>0$. Let $0\le g_R\in L^{\infty}(\1 B_R\times
(0,T))$ and let $\4{u}_R$ be the solution of \eqref{Dirichlet-problem-g}
in $Q_R^T$ with $g=g_R$ and $\Omega=B_R$. Then $\4{u}_R$ converges 
uniformly on every compact subset of $\R^n\times (0,T)$ to the 
unique solution $u$ of \eqref{fast-diff-IVP} as $R\to\infty$.
\end{thm}
\begin{proof}
Let $R_0$, $w_R$, be given by Lemma~\ref{positive-lem} and let $u_R$ 
be as in the proof of Theorem~\ref{existence-finite-time}.
By the maximum principle $u_R\ge\4{u}_R\ge w_{R/5}$ in $Q_R^T$ for any 
$R>5R_0$. Since both $u_R$ and $w_{R/5}$ converges uniformly on every 
compact subset of $\R^n\times (0,T)$ to the unique solution of 
\eqref{fast-diff-IVP} as $R\to\infty$, $\4{u}_R$ converges uniformly 
on every compact subset of $\R^n\times (0,T)$ to the unique solution 
$u$ of \eqref{fast-diff-IVP} in $\R^n\times (0,\infty)$ which satisfies 
\eqref{aronson-benilan} in $\R^n\times (0,\infty)$ as $R\to\infty$
and the theorem follows.
\end{proof}

By the proof of Theorem~\ref{existence-finite-time} and 
Theorem~\ref{existence-finite-time-2}, Theorem~\ref{existence-thm} 
follows. In fact we have the following more general result.

\begin{thm}\label{existence-finite-time-3}
Let $n\ge 3$, $0<m\le (n-2)/n$, and let $p$ satisfy \eqref{p-cond}. 
Let $C_1>0$ be as in Lemma~\ref{positive-lem}. Suppose $0\le u_0\in 
L_{loc}^p(\R^n)$ satisfies \eqref{initial-value-average-infty}. Let 
$0\le g_R\in L^{\infty}(\1 B_R\times (0,\infty))$ and let $\4{u}_R$ 
be the solution of \eqref{Dirichlet-problem-g} in $Q_R^{T_R}$ with $g=g_R$
and $\Omega=B_R$ where $T_R$ is the extinction time of $\4{u}_R$
with $T_R=\infty$ if $\4{u}_R>0$ in $Q_R$. Then $\4{u}_R$ 
converges uniformly on every compact subset of $\R^n\times (0,\infty)$ 
to the unique solution $u$ of \eqref{fast-diff-IVP} in
$\R^n\times (0,\infty)$ which satisfies \eqref{aronson-benilan} in 
$\R^n\times (0,\infty)$ as $R\to\infty$.
\end{thm}

\begin{cor}\label{global-existence-cor}
Let $n\ge 3$, $0<m\le (n-2)/n$, $q<2/(1-m)$, and $p$ satisfy 
\eqref{p-cond}. Suppose $0\le u_0\in L_{loc}^p(\R^n)$ 
satisfies $u_0=\2{u}_0+\phi$ where $0\le\2{u}_0\in L_{loc}^p(\R^n)$ and 
$\phi\in L^1(\R^n)\cap L^p(\R^n)$ such that 
\begin{equation}\label{u0-bar-infty2}
\liminf_{|x|\to\infty}|x|^q\2{u}_0(x)\ge A
\end{equation}
for some constant $A>0$. Then \eqref{fast-diff-IVP} has a unique 
global solution $u$ in $\R^n\times (0,\infty)$ which satisfies 
\eqref{aronson-benilan} in $\R^n\times (0,\infty)$. 
\end{cor}
\begin{proof}
By \eqref{u0-bar-infty2} there exists a constant $R_1>0$ such that
$$
\2{u}_0(x)\ge (A/2)|x|^{-q}\quad\forall |x|\ge R_1.
$$
Since $q<2/(1-m)\le n$,
\begin{align*}
\frac{1}{R^{n-\frac{2}{1-m}}}\int_{B_R}u_0\,dx
\ge&\frac{1}{R^{n-\frac{2}{1-m}}}\int_{B_R}\2{u}_0\,dx
-\frac{1}{R^{n-\frac{2}{1-m}}}\|\phi\|_{L^1}\\
\ge&\frac{\omega_nA}{2(n-q)}(R^{\frac{2}{1-m}-q}-R^{\frac{2}{1-m}-n}R_1^{n-q})
-R^{\frac{2}{1-m}-n}\|\phi\|_{L^1}\quad\forall R>R_1\\
\to&\infty\quad\mbox{ as }R\to\infty.
\end{align*}
Hence by Theorem~\ref{existence-thm} \eqref{fast-diff-IVP} has a unique 
global solution $u$ in $\R^n\times (0,\infty)$ which satisfies 
\eqref{aronson-benilan} in $\R^n\times (0,\infty)$ and the corollary 
follows.
\end{proof}

\begin{cor}\label{existence-global-soln2}
Let $n\ge 3$, $0<m\le (n-2)/n$, $q<2/(1-m)$, and $p$ satisfy 
\eqref{p-cond}. Suppose $0\le u_0\in L_{loc}^p(\R^n)$ satisfies
$u_0(x)\ge C_0/|x|^q$ for any $|x|\ge R_1$ and some constants 
$C_0>0$, $R_1>0$. Then there exists a unique global solution $u$ 
of \eqref{fast-diff-IVP} in $\R^n\times (0,\infty)$ which satisfies 
\eqref{aronson-benilan} in $\R^n\times (0,\infty)$.  
\end{cor}

\begin{lem}\label{psi-upper-bd-lem}
Let $n\ge 3$, $0<m\le (n-2)/n$, and $q<n/p$ for some constant $p$ 
satisfying \eqref{p-cond}. Let $\psi$ be the 
unique solution of \eqref{fast-diff-IVP} in $\R^n\times (0,\infty)$ 
with initial value $A|x|^{-q}$ which satisfies \eqref{aronson-benilan} 
in $\R^n\times (0,\infty)$ given by Corollary~\ref{existence-global-soln2}.
Then
\begin{equation}\label{psi-upper-bd}
\psi(x,t)\le A|x|^{-q}\quad\forall x\ne 0,t>0.
\end{equation}
\end{lem}
\begin{proof}
For any $R>0$, $k>AR^{-q}$, let $v_{R,k}$ be the solution of the problem
\begin{equation*}
\left\{\begin{aligned}
v_t=&\frac{n-1}{m}\Delta v^m\qquad\quad\mbox{ in }Q_R\\
v=&AR^{-q}\quad\qquad\qquad\mbox{ on }\1 B_R\times (0,\infty)\\
v(x,0)=&\min (A|x|^{-q},k)\quad\mbox{ in }B_R.
\end{aligned}\right.
\end{equation*}
Then by the maximum principle (Lemma 2.3 of \cite{DaK}), 
\begin{equation}\label{vr-lower-bd-2}
AR^{-q}\le v_{R,k}\le v_{R,k'}\le k'\quad\mbox{ in }Q_R
\quad\forall k'>k>AR^{-q},R>0.
\end{equation}  
By Theorem~\ref{L-infinity-initial-Lp-bd2} for any $t_2>t_1>0$, $R>0$,
there exists a constant $M>0$ such that
\begin{equation}\label{vr-upper-bd-2}
v_{R,k}\le M\quad\mbox{ in }\2{B}_R\times [t_1,t_2]
\quad\forall k>AR^{-q}.
\end{equation}
By \eqref{vr-lower-bd-2}, \eqref{vr-upper-bd-2}, and an argument similar
to the proof of Theorem~\ref{existence-thm}, $v_{R,k}$ increases and 
converges uniformly on every compact subset of $\2{B}_R\times (0,\infty)$ 
as $k\to\infty$ to the solution $v_R$ of 
\begin{equation*}
\left\{\begin{aligned}
v_t=&\frac{n-1}{m}\Delta v^m\quad\mbox{ in }Q_R\\
v=&AR^{-q}\qquad\quad\mbox{ on }\1{B}_R\times (0,\infty)\\
v(x,0)=&A|x|^{-q}\qquad\,\,\,\mbox{ in }B_R.
\end{aligned}\right.
\end{equation*}
Let $t_2>t_1>0$. For any $k_0>AR^{-q}$ we choose $\3_0>0$ such that 
$k_0<A\3_0^{-q}$. Then by the maximum principle,
\begin{align}
&v_{R,k}\le A|x|^{-q}\quad\forall\3\le|x|\le R,t>0,k\ge k_0,0<\3<\3_0\notag\\
\Rightarrow\quad&v_R\le A|x|^{-q}\quad\forall 0<|x|\le R,t>0\quad\mbox{ as }
k\to\infty.\label{vr-ar-q-bd}
\end{align}
By Theorem~\ref{existence-finite-time-3} $v_R$ converges to $\psi$ uniformly
on every compact subset of $\R^n\times (0,\infty)$ as $R\to\infty$. 
Letting $R\to\infty$ in \eqref{vr-ar-q-bd} we get \eqref{psi-upper-bd}
and the lemma follows.
\end{proof}

\begin{thm}\label{existence-self-similar-soln}
Let $n\ge 3$, $0<m\le (n-2)/n$, and $q<n/p$ for some constant $p$ satisfying
\eqref{p-cond}. Let $\alpha$, $\beta$, be given by \eqref{alpha-beta-defn}.
Then there exists a unique radially symmetric self-similar solution $\psi$ of 
\eqref{fast-diff-IVP} in $\R^n\times (0,\infty)$ with initial value 
$A|x|^{-q}$ which satisfies \eqref{aronson-benilan} and 
\eqref{psi-upper-bd} in $\R^n\times (0,\infty)$ and
\begin{equation}\label{self-similar-eqn}
\psi(x,t)=\gamma^q\psi(\gamma x,\gamma^{\frac{1}{\beta}}t)\quad\forall 
x\in\R^n,t>0,\gamma>0.
\end{equation}
In particular,
\begin{equation}\label{self-similar-eqn-2}
\psi(x,t)=t^{-\alpha}\psi (t^{-\beta}x,1)\quad\forall x\in\R^n,t>0.
\end{equation}
Moreover $\4{v}(x)=\psi (x,1)$ satisfies \eqref{elliptic} in $\R^n$.
If in addition $0<m<(n-2)/n$, then $\psi\in 
C((\R^n\setminus\{0\})\times (0,\infty))$ and 
\begin{equation}\label{psi-x-infty}
|x|^q\psi(x,t)\to A\quad\mbox{ uniformly on }[0,T]\quad\mbox{ as }
|x|\to\infty.
\end{equation}
\end{thm}
\begin{proof}
By Corollary~\ref{existence-global-soln2} and Lemma~\ref{psi-upper-bd-lem}
there exists a unique solution $\psi$ of \eqref{fast-diff-IVP} in 
$\R^n\times (0,\infty)$ with initial value $A|x|^{-q}$ which satisfies 
\eqref{aronson-benilan} and \eqref{psi-upper-bd} in $\R^n\times (0,\infty)$. 
Since $\psi (F(x),t)$ is also a solution of 
\eqref{fast-diff-eqn} in $\R^n\times (0,\infty)$ with initial 
value $A|x|^{-q}$ for any rotation $F:\R^n\to\R^n$, by uniqueness of 
solution,  
$$
\psi (x,t)=\psi(F(x),t)\quad\forall x\in\R^n,t\ge 0.
$$
Hence $\psi (x,t)$ is radially symmetric in $x$.
For any $\gamma>0$, let $\psi_{\gamma}(x,t)=\gamma^q\psi
(\gamma x,\gamma^{\frac{1}{\beta}}t)$ where $\beta$ is given by 
\eqref{alpha-beta-defn}. Then $\psi_{\gamma}$ also satisfies 
\eqref{fast-diff-eqn} in $\R^n\times (0,\infty)$ with initial
value $A|x|^{-q}$. By the uniqueness of solution, 
\begin{equation*}
\psi(x,t)=\psi_{\gamma}(x,t)\quad\forall 
x\in\R^n,t>0,\gamma>0
\end{equation*}
and \eqref{self-similar-eqn} follows. Hence $\psi$ is the unique radially 
symmetric self-similar solution of 
\eqref{fast-diff-IVP} in $\R^n\times (0,\infty)$ with initial value 
$A|x|^{-q}$ which satisfies \eqref{aronson-benilan} in $\R^n\times 
(0,\infty)$. Putting $\gamma=t^{-\beta}$ in \eqref{self-similar-eqn},
\begin{equation}\label{v-tilde-eqn}
\psi(x,t)=t^{-\alpha}\psi (t^{-\beta}x,1)=t^{-\alpha}\4{v}(t^{-\beta}x)
\quad\forall x\in\R^n,t>0
\end{equation}
and \eqref{self-similar-eqn-2} follows. Substituting \eqref{v-tilde-eqn} 
into \eqref{fast-diff-eqn}, we get that $\4{v}$ satisfies \eqref{elliptic} 
in $\R^n$. 

We now let $0<m<(n-2)/n$. We will use a modification of the proof of 
Corollary 2.8 and Corollary 2.9 of \cite{Hu1} to show that $\psi$ 
satisfies
\begin{equation}\label{psi-bdary-value}
\psi (|x_0|,t)=\psi (x_0,t)\to A|x_0|^{-q}\quad\mbox{ as }t\to 0\quad
\forall x_0\ne 0.
\end{equation}
We will first compare 
$\psi$ with the Barenblatt solution $B_k$ given by \eqref{Barenblatt} for some 
constants $T>0$, $k>0$, to be determined later. Since by \eqref{p-cond}
$q<n/p<2/(1-m)$, by the Young inequality,
\begin{align}\label{xq-ineqn}
|x|^q=&(T^{\frac{1}{n-2-nm}}|x|)^qT^{-\frac{q}{n-2-nm}}\notag\\
\le&(q(1-m)/2)(T^{\frac{1}{n-2-nm}}|x|)^{\frac{2}{1-m}}
+(1-(q(1-m)/2))T^{-\frac{2\alpha}{n-2-nm}}\notag\\
\le&q(1-m)(k+T^{\frac{2}{n-2-nm}}|x|^2)^{\frac{1}{1-m}}
\end{align}
where $\alpha$ is given by \eqref{alpha-beta-defn} and 
$k=(2q^{-1}(1-m)^{-1}-1)^{1-m}T^{-\frac{2\alpha(1-m)}{n-2-nm}}$.
Let 
$$
T=\left(\frac{A}{q(1-m)C_{\ast}^{\frac{1}{1-m}}}\right)^{\frac{n-2-mn}{n}}
$$ 
where
$C_{\ast}=2(n-1)(n-2-nm)/(1-m)$. Then by \eqref{xq-ineqn},
\begin{equation}\label{u0-Bk-compare}
B_k(x,0)\le A|x|^{-q}\quad\forall x\in\R^n.
\end{equation}
Let $u_R$ be the solution of \eqref{infinity-bd-problem} in $Q_R$
with $\Omega=B_R$ and initial value $A|x|^{-q}$. By \eqref{u0-Bk-compare} 
and the maximum principle (\cite{Hu2}),
\begin{equation}\label{ur-Bk-compare}
u_R(x,t)\ge B_k(x,t)\quad\mbox{ in }Q_R.
\end{equation}
Since by the proof of Theorem~\ref{existence-finite-time} $u_R$ converges 
uniformly on every compact subset of $\R^n\times (0,\infty)$ to $\psi$ 
as $R\to\infty$, letting $R\to\infty$ in \eqref{ur-Bk-compare},
\begin{equation}\label{psi-Bk-compare}
\psi(x,t)\ge B_k(x,t)\quad\forall x\in\R^n,t>0.
\end{equation}
For any $x_0\ne 0$, let $R_1=|x_0|/2$. By \eqref{psi-upper-bd} and
\eqref{psi-Bk-compare} there exist constants $C_2>0$, $C_3>0$, such that
\begin{equation}\label{psi-u-l-bd}
C_2\le\psi (x,t)\le C_3\quad\forall x\in B_{R_1}(x_0),t>0.
\end{equation}
Since $A|x|^{-q}$ is continuous for all $x\ne 0$, by \eqref{psi-u-l-bd}
and an argument similar to that of \cite{DaFK} \eqref{psi-bdary-value}
follows.
Hence putting $\gamma=1/|x|$, $x\ne 0$, into \eqref{self-similar-eqn} 
by \eqref{psi-bdary-value} we have
$$
|x|^q\psi (x,t)=\psi (1,t/|x|^{\frac{1}{\beta}})\to A
$$
uniformly on $[0,T]$ as $|x|\to\infty$ for any $T>0$ and the theorem 
follows.
\end{proof}

\section{Asymptotic behaviour of solutions}
\setcounter{equation}{0}
\setcounter{thm}{0}

In this section we will prove the asymptotic large time behaviour of 
the global solution of \eqref{fast-diff-IVP}. 

\noindent{\ni{\it Proof of Theorem~\ref{asymptotic-thm}}:}
We will use a modification of the proof Theorem 2.1 of \cite{Hs1} to prove
the theorem. For any $\gamma\ge 1$, let $u_{\gamma}(x,t)=\gamma^qu(\gamma x,
\gamma^{\frac{1}{\beta}}t)$, $u_{0,\gamma}(x)=\gamma^qu_0(\gamma x)$, $
\2{u}_{0,\gamma}(x)=\gamma^q\2{u}_0(\gamma x)$, and 
$\phi_{\gamma}(x)=\gamma^q\phi(\gamma x)$. Let $C_1>0$ be as in the proof of
Lemma~\ref{positive-lem} and Theorem~\ref{existence-finite-time}. Then
\begin{equation}\label{initial-u-gamma-value}
u_{\gamma}(x,0)=u_{0,\gamma}(x)=\2{u}_{0,\gamma}(x)+\phi_{\gamma}(x).
\end{equation}
By the same computation as the proof of Theorem 2.1 of \cite{Hs1} we get
\begin{equation}\label{phi-lp-bd}
\|\phi_{\gamma}\|_{L^p}\le\|\phi\|_{L^p}\quad\forall\gamma\ge 1, 
\end{equation}
\begin{equation}\label{u0-bar-gamma-lp-limit}
\int_{B_R}\2{u}_{0,\gamma}^p\,dx\le C(1+R^{n-pq})\quad\forall R>0,\gamma\ge 1,
\end{equation}
and
\begin{equation}\label{phi-l1-limit}
\|\phi_{\gamma}\|_{L^1}\le\gamma^{q-n}\|\phi\|_{L^1}\to 0\quad\mbox{ as }
\gamma\to\infty.
\end{equation}
By \eqref{u0-bar-infty} for any $0<\3<A$ there exists a constant 
$R_{\3}>0$ such that
\begin{align}\label{u0-gamma-bar-u-l-bd}
&(A-\3)|x|^{-q}\le\2{u}_0(x)\le(A+\3)|x|^{-q}\quad\forall |x|\ge R_{\3}
\notag\\
\Rightarrow\quad&(A-\3)|x|^{-q}\le\2{u}_{0,\gamma}(x)\le(A+\3)|x|^{-q}
\quad\forall |x|\ge R_{\3}/\gamma.
\end{align}
Hence by \eqref{u0-bar-gamma-lp-limit} and \eqref{u0-gamma-bar-u-l-bd},
\begin{align}\label{u0-gamma-bar-integral-ineqn}
&\int_{|x|\le R}|\2{u}_{0,\gamma}(x)-A|x|^{-q}|\,dx\notag\\
\le&\3 A\int_{R_{\3}/\gamma\le |x|\le R}|x|^{-q}\,dx
+C(R_{\3}/\gamma)^{n(1-\frac{1}{p})}
\left(\int_{|x|\le R_{\3}/\gamma}\2{u}_{0,\gamma}^p\,dx\right)^{\frac{1}{p}}
+A\int_{|x|\le R_{\3}/\gamma}|x|^{-q}\,dx\notag\\
\le&\3 A\int_{|x|\le R}|x|^{-q}\,dx
+C(R_{\3}/\gamma)^{n(1-\frac{1}{p})}(1+(R_{\3}/\gamma)^{n-pq})^{\frac{1}{p}}
+A\int_{|x|\le R_{\3}/\gamma}|x|^{-q}\,dx\quad\forall R>0,\gamma>R_{\3}/R.
\end{align}
Letting first $\gamma\to\infty$ and then $\3\to 0$ in 
\eqref{u0-gamma-bar-integral-ineqn}, by \eqref{phi-l1-limit} 
we get
\begin{align}
&\lim_{\gamma\to\infty}\int_{|x|\le R}|\2{u}_{0,\gamma}(x)-A|x|^{-q}|\,dx=0
\quad\forall R>0\notag\\
\Rightarrow\quad&\lim_{\gamma\to\infty}
\int_{|x|\le R}|u_{0,\gamma}(x)-A|x|^{-q}|\,dx=0
\quad\forall R>0.\label{u0-gamma-bar-limit}
\end{align}
By \eqref{initial-u-gamma-value}, \eqref{phi-lp-bd}, 
\eqref{u0-bar-gamma-lp-limit}, and Corollary~\ref{L-infinity-initial-Lp-bd1},
for any constants $R>0$, $t_2>t_1>0$, there exists a constant $C_2>0$ such 
that
\begin{equation}\label{u-gamma-upper-bd}
u_{\gamma}(x,t)\le C_2\quad\forall |x|\le R, t_1\le t\le t_2.
\end{equation}
Let $\delta$, $R_0$, be as in Lemma~\ref{positive-lem}. Let $t_2>t_1>0$. 
Then by \eqref{phi-l1-limit},
\begin{align}\label{u0-integral-lower-bd}
\frac{1}{R^{n-\frac{2}{1-m}}}\int_{|x|\le R}u_{0,\gamma}\,dx
\ge&\frac{1}{R^{n-\frac{2}{1-m}}}\left(\int_{|x|\le R}A|x|^{-q}\,dx
-\int_{|x|\le R}|\2{u}_{0,\gamma}(x)-A|x|^{-q}|\,dx
-\|\phi_{\gamma}\|_{L^1(\R^n)}\right)\notag\\
\ge&\frac{A\omega_n}{n-q}R^{\frac{2}{1-m}-q}-\frac{1}{R^{n-\frac{2}{1-m}}}
\left(\int_{|x|\le R}|\2{u}_{0,\gamma}(x)-A|x|^{-q}|\,dx
+\gamma^{q-n}\|\phi\|_{L^1(\R^n)}\right).
\end{align}
Since $q<n/p<2/(1-m)$, there exists a constant $R_1>R_0$ such that 
\begin{equation}\label{R-power-infty}
\frac{A\omega_n}{n-q}R^{\frac{2}{1-m}-q}
\ge\frac{2C_1t_2^{\frac{1}{1-m}}}{(1-\delta)^{n-\frac{2}{1-m}}}\quad
\forall R\ge R_1.
\end{equation}
Let $R\ge R_1$. By \eqref{u0-gamma-bar-limit} there exists $\gamma_R>1$
such that the last term on the right hand side of 
\eqref{u0-integral-lower-bd} is less than 
\begin{equation}\label{l1-small}
\le C_1t_2^{\frac{1}{1-m}}\quad\forall \gamma\ge\gamma_R.
\end{equation}
Let $|x_0|\le\delta R$. Then by \eqref{u0-integral-lower-bd}, 
\eqref{R-power-infty}, and \eqref{l1-small},
\begin{align}\label{u0-integral-lower-bd2}
\frac{1}{R^{n-\frac{2}{1-m}}}\int_{B_R(x_0)}u_{0,\gamma}\,dx
\ge\frac{1}{R^{n-\frac{2}{1-m}}}\int_{|x|\le(1-\delta)R}u_{0,\gamma}\,dx
\ge C_1t_2^{\frac{1}{1-m}}\quad\forall R\ge\frac{R_1}{1-\delta},
\gamma\ge\gamma_R.
\end{align}
Let $v_{R,\gamma}$ be the weak solution of \eqref{Dirichlet-problem-1} 
in $Q_{5R}^{T_{R,\gamma}}$ with $u_0$ being replaced by $u_{0,\gamma}$
where $T_{R,\gamma}>0$ is the extinction time of $v_{R,\gamma}$.
Let $\4{v}_{R,\gamma}$ be the weak solution of \eqref{Dirichlet-problem-3}
in $Q_{3R}^{\4{T}_{R,\gamma}}(x_0)$ with $u_0$ being replaced by $u_{0,\gamma}$
where $\4{T}_{R,\gamma}>0$ is the extinction time of $\4{v}_{R,\gamma}$. By 
\eqref{u0-bar-gamma-lp-limit} and an argument similar to that on 
P.232-233 of \cite{P} there exists a constant $C>0$ such that
\begin{equation}\label{extinction-time-upper-bd}
\4{T}_{R,\gamma}\le C\left(\int_{B_{2R}}u_{0,\gamma}^p\,dx\right)^{\frac{2}{n}}
\le C'(1+R^{n-pq})^{\frac{2}{n}}\quad\forall R>0,\gamma\ge 1.
\end{equation}
Since by the maximum principle $u_{\gamma}\ge v_{R,\gamma}\ge \4{v}_{R,\gamma}$ 
in $Q_{3R}^{\4{T}_{R,\gamma}}(x_0)$, by \eqref{u0-integral-lower-bd2}, 
\eqref{extinction-time-upper-bd}, 
and an argument similar to the proof of Lemma~\ref{positive-lem}, 
$$
\4{T}_{R,\gamma}>t_2\quad\forall R\ge R_1/(1-\delta), \gamma\ge\gamma_R,
$$
and there exist constants $C_0'>0$, $C_0''>0$, such that
\begin{align}\label{u-gamma-lower-bd}
\inf_{\2{B_{\delta R}}\times [t_1,t_2]}u_{\gamma}^m
\ge&\inf_{\2{B_{\delta R}}\times [t_1,t_2]}v_{R,\gamma}^m\notag\\
\ge&C_0'\4{T}_{R,\gamma}^{-\frac{1}{1-m}}R^{-\frac{2m}{1-m}}t_1^{\frac{m}{1-m}}
(2t_2^{\frac{m}{1-m}}-t_1^{\frac{m}{1-m}})\notag\\
\ge&C_0''(1+R^{n-pq})^{-\frac{2}{n(1-m)}}R^{-\frac{2m}{1-m}}t_1^{\frac{m}{1-m}}
(2t_2^{\frac{m}{1-m}}-t_1^{\frac{m}{1-m}})>0\quad\forall 
R\ge\frac{R_1}{1-\delta},\gamma\ge\gamma_R.
\end{align}
By \eqref{u-gamma-upper-bd} and \eqref{u-gamma-lower-bd} the equation
\eqref{fast-diff-eqn} for the family of functions 
$\{u_{\gamma}:\gamma\ge\gamma_R\}$ are uniformly parabolic on 
$\2{B_{\delta R}}\times [t_1,t_2]$ for any $R\ge R_1/(1-\delta)$. By the 
parabolic Schauder estimates \cite{LSU} the family of functions 
$\{u_{\gamma}:\gamma\ge\gamma_R\}$ are equi-Holder continuous on 
$\2{B_{\delta R}}\times [t_1,t_2]$ for any $R\ge R_1/(1-\delta)$. 

Let $\{\gamma_i\}_{i=1}^{\infty}$ be a sequence such that $\gamma_i\ge 1$ 
for all $i\in\Z^+$ and $\gamma_i\to\infty$ as $i\to\infty$. Then by the 
Ascoli Theorem and a diagonalization argument the sequence
$\{u_{\gamma_i}\}_{i=1}^{\infty}$ has a subsequence which we may assume 
without loss of generality to be the sequence itself that converges
uniformly on every compact subset of $\R^n\times (0,\infty)$ as
$i\to\infty$ to some solution $\psi$ of \eqref{fast-diff-eqn} that
satisfies \eqref{aronson-benilan} in $\R^n\times (0,\infty)$. 
We will now prove that $\psi$ has initial value $A|x|^{-q}$. For any 
$i,j\in\Z^+$, $R>0$ , $t>0$, 
\begin{align}\label{psi-uj-compare}
&\int_{|x|\le R}|\psi(x,t)-A|x|^{-q}|\,dx\notag\\
\le&\int_{|x|\le R}|\psi(x,t)-u_{\gamma_j}(x,t)|\,dx
+\int_{|x|\le R}|u_{\gamma_j}(x,t)-u_{0,j}|\,dx
+\int_{|x|\le R}|u_{0,j}-A|x|^{-q}|\,dx.
\end{align}
Then similar to \eqref{l1-compare} we have
\begin{align}\label{ui-uj-compare}
&\left(\int_{|x|\le R}|u_{\gamma_i}-u_{\gamma_j}|(x,t)\,dx\right)^{1-m}
\le\left(\int_{|x|\le 2R}|u_{0,\gamma_i}-u_{0,\gamma_j}|\,dx\right)^{1-m}
+CR^{n(1-m)-2}t\notag\\
\Rightarrow\quad&\int_{|x|\le R}|u_{\gamma_i}-u_{\gamma_j}|(x,t)\,dx
\le C'\left(\int_{|x|\le 2R}|u_{0,\gamma_i}-u_{0,\gamma_j}|\,dx
+R^{n-\frac{2}{1-m}}t^{\frac{1}{1-m}}\right)\quad\forall i,j\in\Z^+, 
R>0,t>0.
\end{align}
Letting $i\to\infty$ in \eqref{ui-uj-compare}, by 
\eqref{u0-gamma-bar-limit} we get
\begin{equation}\label{psi-uj-compare2}
\int_{|x|\le R}|\psi-u_{\gamma_j}|(x,t)\,dx
\le C'\left(\int_{|x|\le 2R}|A|x|^{-q}-u_{0,\gamma_j}|\,dx
+R^{n-\frac{2}{1-m}}t^{\frac{1}{1-m}}\right)\quad\forall j\in\Z^+, 
R>0,t>0.
\end{equation}
By \eqref{psi-uj-compare} and \eqref{psi-uj-compare2},
\begin{align}\label{psi-uj-compare3}
&\int_{|x|\le R}|\psi(x,t)-A|x|^{-q}|\,dx\notag\\
\le&\int_{|x|\le R}|u_{\gamma_j}(x,t)-u_{0,j}|\,dx
+C\int_{|x|\le 2R}|u_{0,\gamma_j}-A|x|^{-q}|\,dx+CR^{n-\frac{2}{1-m}}t^{\frac{1}{1-m}}
\quad\forall j\in\Z^+,R>0,t>0.
\end{align}
Letting first $t\to 0$ and then $j\to\infty$ in \eqref{psi-uj-compare3},
by \eqref{u0-gamma-bar-limit} we get
\begin{equation*}
\lim_{t\to 0}\int_{|x|\le R}|\psi(x,t)-A|x|^{-q}|\,dx\quad\forall
R>0.
\end{equation*}
Hence $\psi$ has initial value $A|x|^{-q}$. By 
Theorem~\ref{existence-self-similar-soln} $\psi$ is the unique
self-similar radially symmetric solution of \eqref{fast-diff-eqn} in
$\R^n\times (0,\infty)$ with initial value $A|x|^{-q}$. Since the sequence 
$\{\gamma_i\}_{i=1}^{\infty}$ is arbitrary, $u_{\gamma}$ converges
uniformly to $\psi$ as $\gamma\to\infty$ on every compact 
subset of $\R^n\times (0,\infty)$. In particular 
\begin{equation}\label{u-gamma-limit}
u_{\gamma}(x,1)=\gamma^qu(\gamma x,\gamma^{\frac{1}{\beta}})\to\psi(x,1)
\end{equation}
uniformly on every compact subset of $\R^n\times (0,\infty)$ as
$\gamma\to\infty$. By \eqref{u-gamma-limit},
\begin{equation*}
v(x,t)=t^{\alpha}u(t^{\beta}x,t)\to\psi(x,1)
\end{equation*}
uniformly on every compact subset of $\R^n\times (0,\infty)$ as
$t\to\infty$ and the theorem follows.
\hfill$\square$

\vspace{6pt}

\begin{thebibliography}{99}

\bibitem[A]{A} D.G.~Aronson, {\em The porous medium equation}, CIME
Lectures, in Some problems in Nonlinear Diffusion, Lecture
Notes in Mathematics 1224, Springer-Verlag, New York, 1986. 

\bibitem[AC]{AC} D.G.~Aronson and L.A.~Caffarelli, {\em The initial
trace of a solution of the porous medium equation}, Trans. Amer. Math. 
Soc. 280 (1983), no. 1, 351--366.

\bibitem[BBDGV]{BBDGV} A.~Blanchet, M.~Bonforte, J.~Dolbeault, G.~Grillo
and J.L.~Vazquez, {\em Asymptotics of the fast diffusion equation via 
entropy estimates}, Arch. Rat. Mech. Anal. 191 (2009), 347--385.

\bibitem[BV]{BV} M.~Bonforte and J.L.~Vazquez, {\em Positivity, local 
smoothing, and Harnack inequalities for very fast diffusion equations},
Advances in Math. 223 (2010), 529--578.

\bibitem[BC]{BC} P.~B\'enlian and M.G.~Crandall, {\em Regularizing 
effects of homogenous evolution equations}, pp. 23--39, in 
Contributions to Analysis and Geometry (suppl. to Amer. J. Math.), 
Johns Hopkins Univ. Press, Baltimore, MD, 1981. 

\bibitem[DaFK]{DaFK} B.E.J.~Dahlberg, E.~Fabes and C.E.~Kenig, {\em A
Fatou thoerem for solutions of the porous medium equations}, Proc. 
Amer. Math. Soc. 91 (1984), 205--212.

\bibitem[DaK]{DaK} B.E.J.~Dahlberg and C.E.~Kenig, {\em Nonnegative 
solutions to the generalized porous medium equation}, Rev. Mat.
Iberoamericana 2 (1986), 267--305.

\bibitem[DK]{DK} P.~Daskalopoulos and C.E.~Kenig, {\em Degenerate 
diffusion-initial value problems and local regularity theory}, 
Tracts in Mathematics 1, European Mathematical Society, 2007.

\bibitem[DP]{DP} P.~Daskalopoulos and M.~Del Pino, {\em On Nonlinear
parabolic equations of very fast diffusion}, Arch. Rat. Mech. Anal. 137 
(1997), 363--380.

\bibitem[DS1]{DS1} P.~Daskalopoulos and N.~Sesum, {\em On the extinction
profile of solutions to fast diffusion}, J. Reine Angew Math. 622 (2008),
95--119.

\bibitem[DS2]{DS2} P.~Daskalopoulos and N.~Sesum, {\em The classification of
locally conformally flat Yamabe solitons}, http://arxiv.org/abs/1104.2242.

\bibitem[HP]{HP} M.A.~Herrero and M.~Pierre, {\em The Cauchy problem for 
$u_t=\Delta u^m$ for $0<m<1$}, Trans. Amer. Math. Soc. 291 (1985), 
no. 1, 145--158.

\bibitem[Hs1]{Hs1} S.Y.~Hsu, {\em Large time behaviour of solutions of a
singular diffusion equation in $\R^n$}, Nonlinear Analysis TMA 62 (2005), 
no. 2, 195--206.

\bibitem[Hs2]{Hs2} S.Y.~Hsu, {\em Singular limit and exact decay rate 
of a nonlinear elliptic equation}, http://arxiv.org/abs/1107.2735v1.

\bibitem[Hu1]{Hu1} K.M.~Hui, {\em On some Dirichlet and Cauchy problems 
for a singular diffusion equation}, Differential Integral Equations 15 
(2002), no. 7, 769¡V804.

\bibitem[Hu2]{Hu2} K.M.~Hui, {\em Singular limit of solutions of the very
fast diffusion equation}, Nonlinear Anal. TMA 68 (2008), 1120--1147.

\bibitem[K]{K} T.~Kato, {\em Schr\"odinger operators with singular potentials},
Israel J. Math. 13 (1973), 135--148. 

\bibitem[LSU]{LSU} O.A.~Ladyzenskaya, V.A.~Solonnikov and
N.N.~Uraltceva, {\em Linear and quasilinear equations of parabolic type}, 
Transl. Math. Mono. vol. 23, Amer. Math. Soc., Providence, R.I., U.S.A., 
1968.

\bibitem[P]{P} L.A.~Peletier, {\em The porous medium equation in 
Applications of Nonlinear Analysis in the Physical Sciences}, 
H.Amann, N.Bazley, K.Kirchgassner editors, Pitman, Boston, 1981.

\bibitem[PS]{PS} M.~Del Pino and M.~S\'aez, {\em On the extinction profile 
for solutions of $u_t=\Delta u^{(N-2)/(N+2)}$}, Indiana Univ. Math. J. 50
(2001), no. 1, 611--628.

\bibitem[R]{R} G.~Rosen, {\em Nonlinear heat conduction in solid $H_2$}, 
Rhys. Rev. B 19 (1979), 2398--2399. 

\bibitem[V1]{V1} J.L.~Vazquez, {\em Nonexistence of solutions for nonlinear
heat equations of fast-diffusion type}, J. Math. Pures Appl. 71 (1992), 
503--526. 

\bibitem[V2]{V2} J.L.~Vazquez, {\em Smoothing and decay estimates for nonlinear
diffusion equations}, Oxford Lecture Series in Mathematics and its Applications
33, Oxford University Press, Oxford, 2006.

\bibitem[V3]{V3} J.L.~Vazquez, {\em The porous medium equation-Mathematical 
Theory}, Oxford Mathematical Monographs, Oxford University Press, 2007.

\end{thebibliography}
\end{document}